\documentclass[final]{siamltex}
\usepackage{amssymb,latexsym,amsmath,graphics,setspace}
\usepackage{epsfig,epsf,amsfonts}
\usepackage{color,multicol,colordvi,graphicx,multirow}
\usepackage{amsfonts, epsfig}
\usepackage{hyperref}
\renewcommand{\thepage}{\hss\thesection-\arabic{page}\hss}
\normalbaselineskip
\newcommand{\fb}{f_{\textrm{\begin{tiny}B\end{tiny}}}}
\newtheorem{remark}{Remark}
\newcommand{\We}{\mathcal W_{\textrm{\begin{tiny}\textrm E\end{tiny}}}}
\newcommand{\Potential}{\mathcal{W}}
\newcommand{\wldg}{\Potential_{\textrm{\begin{tiny}{LdG}\end{tiny}}}}
\newcommand{\wbtw}{\Potential_{\textrm{\begin{tiny}{BTW}\end{tiny}}}}
\newcommand{\bx}{\boldsymbol{x}}
\newcommand{\bX}{\boldsymbol{X}}
\newcommand{\be}{\boldsymbol{e}}
\newcommand{\XX}{\boldsymbol{X}}
\newcommand{\tr}{{\textrm{tr}}}
\newcommand{\bvphi}{\boldsymbol{\varphi}}
\newcommand{\bn}{\boldsymbol{n}}

\newcommand{\wconv}{\rightharpoonup}
\newcommand{\adj}{\operatorname{adj}}

\begin{document}
\title{ Liquid crystal elastomers  and phase transitions in rod networks}
 \author{M.~Carme Calderer \thanks{School of Mathematics, University of Minnesota, 
206 Church Street S.E., Minneapolis, MN 55455, USA.({\tt calde14@.umn.edu}).}
\and
 Carlos  A.Garavito Garz{\'o}n\thanks{School of Mathematics, University of Minnesota, 
206 Church Street S.E., Minneapolis, MN 55455, USA.({\tt garav007@.umn.edu}).}
\and
Chong Luo  \thanks{School of Mathematics, University of Minnesota, 
206 Church Street S.E., Minneapolis, MN 55455, USA.
  ({\tt luochong@gmail.com}).}
}
\maketitle
\pagenumbering{arabic}

\begin{abstract}
In this article, we construct and analyze models of anisotropic crosslinked polymers employing tools from the  theory of  liquid crystal elastomers.  The anisotropy of these systems stems from the  presence of rigid-rod molecular units in the network. 
We study   minimization of the energy  for incompressible as well as compressible materials, combining methods of isotropic nonlinear elasticity with the theory of lyotropic liquid crystals. We apply our results to the study of phase transitions in networks of rigid rods, in order to model the behavior of actin filament systems found in the cytoskeleton.

\end{abstract}




\section{Introduction}

Cytoskeletal networks consist of rigid, rod-like  actin protein units jointed by  flexible crosslinks,  presenting coupled orientational and deformation effects   analogous to liquid crystal elastomers.  The alignment properties of the rigid rods influence the mechanical response of the network to applied stress and deformation, affecting  functionality of the systems \cite{wagner-cytoskeletal2006}, \cite{Gardel2004}.  Parameters that characterize these networks include the aspect ratio of the rods and the average length of the crosslinks, with a large span of  parameter values found across in-vivo networks. For instance,   cytoskeletal networks of  red blood cells have very large linkers and small rod aspect ratio \cite{Richieri-Akeson1985}, \cite{Gamez2008}, whereas those of cells of  the  outer hair  of the ear have very large aspect ratio and  short linkers favoring well aligned nematic, in order to achieve an  efficient sound propagation \cite{JerryOHC1995}.

This article is motivated  by  the works on 
Montecarlo simulations of phase transitions in rigid rod fluids by Bates el al.  \cite{BATES} and  the later application to actin networks by Dalhaimer et al. \cite{dalhaimer2007crosslinked}.   
In these articles, the authors  discuss experimentally observed alignment states and their phases transitions as well as  predictions from  numerical experiments. 
 They report on  a wide range of anisotropic regimes,  including the  uncross-linked  fluid network, in the nematic as well as the isotropic state, and the crystal-glass states involving
elastomer microstructure.  A goal of our work  is to obtain a continuum model matching  predictions of  the molecular simulations and available experiments.

A nematic fluid consists of  interacting rod-like molecules that have the tendency to aling along preferred directions and  the ability to flow under applied forces.
Liquid crystal elastomers are  anisotropic nonlinear elastic
materials, with the source of anisotropy stemming from  elongated, rigid monomer side
groups, or from main chain rod-like elements. They  are elastic solids that may also present fluid regimes of {\cite{conti2002soft}}, (\cite{desimone-doltzman2000}, \cite{desimone-doltzman2002}, \cite{fried-sellers2006}, \cite{warner2007liquid}). 
The interaction between the rod units and the network is at the core of liquid crystal elastomer behavior. In {\it main chain } 
elastomers, the  connected rigid units are part of  the backbone chains of the system and in {\it side-chain} elastomers, the rod
 units are attached to the polymer backbone.
  In both cases, the backbone chains are crosslinked into a network. 
Models of anisotropic  polymer melts and their  non-Newtonian  behavior have received significant attention (\cite{Forest2004}, \cite{Rey2010}, \cite{Rey2007}).

Ordering in nematic fluids is affected by temperature in thermotropic liquid crystals and by rod concentration in lyotropic ones.  
At high temperature or low concentration, respectively, nematic fluids are found in the isotropic state, experiencing a transition 
to the nematic upon cooling the thermotropic liquid or increasing the rod concentration of the lyotropic \cite{longa1986}.  In  rod-like systems, such as actin fiber networks, the phase transition behavior is affected by the density of rods. 

We consider anisotropic systems such that the total energy is the sum of the Landau-de Gennes liquid crystal energy of the nematic and an anisotropic elastic 
stored energy function. 
This energy involves two sources of anisotropy expressed by symmetric second order tensors,  that associated with  the rigid units, 
represented by the  nematic order tensor $Q$, and that of the network described by the positive definite, step-length tensor $L$.  
The tensor $L$ encodes the shape of the network:  it is spherical for isotropic polymers and spheroidal for uniaxial nematic 
elastomers, and has eigenvalues $l_\|$ and $l_{\perp}$ (double). 
The quantity $r:= \frac{l_{\|}}{l_\perp}-1$ measures the degree of anisotropy of the network, with positive values corresponding to  
prolate systems and negative ones to the oblate shapes. In the prolate geometry, the eigenvector $\bn$ 
associated with  $l_{\|}$ is the director of the theory, giving the average direction of alignment of the rods and also 
the direction of shape elongation of the network. It is natural to assume that $L$ and $Q$ share eigenvectors. In particular we 
assume that they are linearly related,  so that  for $L$ prescribed, we take $Q$ as its traceless version, that is $Q= L- \frac{1}{3}\tr L I$ (\cite{warner2007liquid}, page 49).
The free energy may also carry information on the anisotropy $L_0$   imprinted in the network  at crosslinking the original polymer melt. 

The Landau-de Gennes free energy density is the sum of scalar quadratic terms of $\nabla Q$ and the bulk  scalar function $f(Q)$. In the de 
Gennes-Landau theory, $f$  is a polynomial function of the trace of powers of $Q$  and describes  the phase transition between the 
isotropic and the nematic  phases \cite{longa1986}.  However, the polynomial growth is not physically realistic since 
it is expected that an infinite energy should be required to reach limiting alignment configurations  \cite{calderer-liu2000}, \cite{ericksen1991liquid}. This turns out to be as well an essential element of our analysis. 
A cautionary note about notation: we will employ the  common symbol $f$  to denote the bulk nematic energy density in the different cases that we address. 

Denoting $F$ the deformation gradient, the elastic energy density proposed by
Blandon, Terentjev and Warner  is  $|L^{-\frac{1}{2}}FL_0^{\frac{1}{2}}|^2$. It is the analog of the Neo-Hookean energy of isotropic 
elasticity, and  also derived from Gaussian statistical mechanics. Taking into account the relevant role played by  the tensor 
$G:={(L^{-1}FF^TL_0)}^{\frac{1}{2}}$ in the trace form of the energy, and  motivated by the theory of existence of minimizers of 
isotropic nonlinear elasticity (\cite{ball1976convexity}), we consider  polyconvex stored energy density functions 
$w(\XX)=\hat w(G(\XX))$, $\XX\in\Omega$. That is, functions $\hat w$ such that there exists a convex function $\Psi$ of the invariants $\{I_i\}_{i=1,2,3}$ of $GG^T$ satisfying $\hat w(G)= \Psi(I_1, I_2, I_3)$. However, since $G$ is not a gradient, we must be able to recover the limiting deformation gradient $F^*$ from the minimizing sequences $\{G_k\}_{k\geq 1}$.  For this, it is necessary that the  minimizing sequences $\{L_k\}$ yield a nonsingular limit.  This is achieved,  by either appropriately regularizing the problem so that  the range of the eigenvalues of $Q$ in the  admissible  set is strictly greater than  $-\frac{1}{3}$, or by requiring the blow up of  $f(Q)$ at the minimum eigenvalue limit, that is, $f(Q)\to \infty $ as $\det(Q+\frac{1}{3}I)\to 0$.

In the case of compressible networks, we further assume that expansion and compression are coupled with order, so that the bulk free energy is now $f(Q, \det F)$. Following the analogous assumptions of isotropic elasticity, we  require that,  for each symmetric traceless tensor $Q$,  $f$  becomes unbounded as $\det F\to \{0, \infty\}$.  We argue that the coupling between  expansion and compression with nematic order is qualitatively analogous to that of lyotropic uniaxial nematic liquid crystals, as proposed by Kuzuu and Doi \cite{kuzuu-doi1983}. In this case, the bulk energy $f(s)$ is parametrized by the rod concentration of the nematic fluid. At low concentration, the isotropic minimum dominates, with  nematic becoming the  preferred phase as the concentration increases.  In the application to rigid rod networks of section 4,  two parameter rates emerge as very relevant:  $\chi=\frac{L_a}{L_x}$, where $L_a$ denotes the typical length of  a  cylindrical rod,  and $L_x$  that of a  cross-linker filament, and the aspect ratio   $A_a=\frac{L_a}{D_a}$ of the  rod where $D_a$ denotes a typical diameter. We assume that $f$ depends on  $s$ and the rigid rod density $\rho$,  and it is also parametrized by the ratio $\chi$. Specifically, following the denominations of {\it loose}, {\it semiloose} and {\it tight} for networks with small through large values of $\chi$, we assume that   $f$ evolves from a function with a single isotropic well for $\chi$ small (large linkers), to having a single nematic well for large $\chi$ (short linkers), presenting an intermediate double-well region.
We also assume that the liquid energy scales according to the aspect ratio of the rods,  resulting in larger nematic contribution with increasing aspect ratio.   Proposition 4.1 summarizes the results on phase transitions under three-dimensional expansion. In subsection 4.2.1, we construct a bulk free energy density with the previously described properties and present  results on numerical simulations  of the phase transition behavior under plane extensions, plots of  phase diagrams in the density-aspect ratio plane, and the graphs of the  equilibrium order parameter $s$ with respect to the rod density. In particular, we find oblate equilibrium states for small values of the aspect ratio, corresponding to disk-like molecules.

In addition to the trace models of liquid crystal elastomer energy studied by Terentjev and Warner (\cite{warner2007liquid} and references therein),  generalizations of these earlier forms have been proposed and studied by several authors (\cite{anderson-carlson-fried1999} and \cite{fried-sellers2003};   \cite{desimone-agostiniani2012}, \cite{cesana2008strain}, \cite{desimone-cesana2011} and \cite{desimone2009elastic}). In the first two references, the authors propose energies based on powers of the earlier trace form, including Ogden type energies,  and study their extensions to account for semisoft elasticity.  Articles by de Simone et al. also propose and study Ogden type energies.  Furthermore, the analysis of equilibrium states presented in \cite{desimone-agostiniani2012} applies to  elastomer energy density functions that  are not quasiconvex. (For instance, these are appropriate to model crystal-like phase transitions.)  Their methods of proof combine the construction of lower quasiconvex envelops, the rigidity  theorem \cite{FrieseckeJamesMuller2002} and tools from the theory of $\Gamma$-convergence. 
Our results apply to a more restrictive class of energy density functions, that is,  polyconvex functions with respect to  the anisotropic deformation tensor $G$.  Our  methods of proof  use tools of isotropic nonlinear elasticity, and as such, are  directly tailored to treating polyconvexity. Moreover,  this approach readily applies to modeling the  nonconvexity associated with  nematic liquid order in networks and the corresponding phase transitions, although it does not cover the more general type of transitions  linked to quasiconvexity. 

This article is organized as follows.   Section 2 is devoted to  modeling, which includes incompressible and compressible liquid crystal elastomers as well as a rod-fluid model. Section 3 is devoted to the minimization of the energy in the different cases. Section 4 presents a study of density dependent liquid crystal phase transitions, with figures corresponding to the phase transition diagram and the order properties with respect to mechanical extension of the system.  The conclusions are described in Section 5.  

Some of the results of section 3 follow from the Ph.D thesis dissertation by Chong Luo
 \cite{chongluo2010}.


\section{The  Landau-de Gennes  liquid crystal elastomer} 

Equilibrium states of nematic liquid crystal elastomers are characterized by the
gradient of deformation tensor $F$ together with the symmetric tensors $L$ and $Q$, 
  describing the shape of the network and the nematic order, respectively.   
  
   We let the open and bounded domain
$\Omega\subset {\mathbf R}^3$ denote the reference configuration of the elastomer. We  denote $L_0$
and $Q_0$, the network anisotropy and the nematic order, respectively. in the
reference configuration. For synthetic elastomers, these tensors model the cross-linking,
in 
fiber networks, these
quantities  represent the anisotropy and order in a relevant state, for instance, the  stress free
state if one exists.
We denote the deformation map of the polymer and its gradient as 

\begin{eqnarray}
&&\bvphi: \Omega\longrightarrow \ \bar\Omega, \quad \bx=\bvphi(\XX),
 \label{deformation} \\ && F=\nabla\bvphi,
  \quad \det F>0.\nonumber 
\end{eqnarray}
In order to understand the relationship among them and how they enter in the energy, we start with a brief survey of these 
tensors at the molecular level, the following the treatment in \cite{warner2007liquid}. 

\subsection{Statistical mechanics of anisotropic polymers}
We now focus on the statistical treatment of single ideal chains.  Let us consider a freely jointed chain composed of $N$ segments of length $a$, and let $\mathbf R$ denote the end-to-end vector of a chain.  The chain follows a random walk with step length $a$. The average end-to-end distance is given by 
\begin{equation}
<|\mathbf R|^2>= N a^2=a l, \quad <R_iR_j>=\frac{1}{3}\delta_{ij}al, \,\, 1\leq i, j\leq 3,\label{average}
\end{equation}
where $l=Na$ is the arc length of the chain, and $<\cdot>$ denotes the ensemble average. 
The probability of a given chain conformation to have an end-to-end vector $\mathbf R$  is the Gaussian distribution
\begin{equation}
p_N(\mathbf R)= (\frac{3}{2\pi R_0^2})^{\frac{3}{2}} e^{-\frac{3|\mathbf R|^2}{2R_0^2}},
\end{equation}
characterized by its variance $R_0$. Moreover,  consistency with (\ref{average}) implies that $R_0=al$.  
The partition function,  $$ Z_N(\mathbf R)= p_N(\mathbf R)Z_N $$
gives the number of configurations with end-to-end vector $\mathbf R$, where $Z_N$ is the total number of chain configurations. So, the free energy of a single chain is 
\begin{equation}
\mathcal F= -k_BT\ln Z_n(\mathbf R)= k_BT(\frac{3|\mathbf R|^2}{2 R_0^2})+ C,
\end{equation}
where $C$ is constant. 

 Another measure of the spatial extension of a single chain is the radius of gyration $R_G$.  It is defined as the root mean square of the distance between each segment of the chain and the center of mass. 
In the case that the number of segments $N>>1$, 
$$<R_G^2>\approx \frac{1}{6}N a^2= \frac{1}{6}aL=\frac{1}{6}R_0.$$
So, in the average, the shape of a polymer chain at equilibrium is spherical with  radius $R_G$. 

The average shape of a liquid crystal polymer is that of an ellipsoid, with  step-length tensor $L$, so that the anisotropic analog of the average of end-to-end distance (\ref{average}) is now 
\begin{equation}
<R_iR_j>=\frac{1}{3}lL_{ij}. \label{average-anisotropic}
\end{equation}
Letting 
 $l_1, l_2$ and $l_3$ denote the ellipsoid semi-axes  along directions $\be_i$,
$i=1,2,3$, $|\be_i|=1$, $L$ admits the spectral representation
\begin{equation}
L=\Sigma_{i=1}^3 l_i \be_i\otimes\be_i. \label {L-biaxial}
\end{equation}
We take $l_1=l_2:=l_\perp$  to represent a uniaxial network  giving the
spheroidal representation for L, and denote $l_{\|}:=l_3$ and $\bn:=\be_3$, so that 
 \begin{equation}
  L =(l_{\|}-l_{\perp}) \bn\otimes\bn + l_{\perp}I. \label{step}
\end{equation} 
In the prolate  symmetry corresponding to  main chain polymers, $l_{\|}>l_{\perp}$ in which case the polymer backbone will stretch along the nematic director $\bn$. (The reverse inequality holds in the case of side-chain oblate elastomers).

 The Gaussian distribution of chain conformations generalized to the anisotropic case is
\begin{equation}
p_N(\mathbf R)=[(\frac{3}{2\pi l})^3\frac{1}{\det L}]^{\frac{1}{2}}e^{-\frac{3}{2l}(\mathbf R\cdot L^{-1}\mathbf R)}.
\end{equation}
As in the isotropic case, the affinity property of chain conformations  leads to the anisotropic version of the neo-Hookean energy in the form
\begin{eqnarray}
\wbtw= \mu(F\cdot L^{-1}F),
\end{eqnarray}
where $\mu$ denotes the shear modulus. 
An expression  that includes the shape at crosslinking encoded in the initial step-length tensor $L_0$ is 
\begin{eqnarray}
\wbtw= \mu \tr (L_0F^TL^{-1}F). \label{btw0}
\end{eqnarray}
In general, scalar functions of powers of the tensors $FL_0F^T$ and $F^TL^{-1}F$ are admissible.
  
Following   the property of  freely  joined rods, we assume that  $L$ and $Q$
have common eigenvectors and propose the constitutive relation
\begin{equation}
L=a_0(Q+\frac{1}{3} I), \label{QL}
\end{equation}
where $a_0=\tr L$ is constant.  The linear constitutive equation (\ref{QL}) is analogous to those proposed by  Terentjev and Warner \cite{warner2007liquid} and Eliot and Sellers \cite{fried-sellers2003} stating that, given a symmetric and traceless tensor $Q$ and a  constant $\beta>0$, there is a one  $\alpha$-parameter family of step-length tensors $L$ with $\tr L=\beta$,  and such that
\begin{equation}
L= \beta(\alpha Q+ \frac{1}{3}I).  \label{QL-1}
\end{equation}
The form (\ref{QL}) corresponds to taking $\alpha=1$ and $\beta=a_0$ in (\ref{QL-1}). 

In order to interpret this condition,  and following the approach in \cite{majumdar2010landau},  we  appeal to the spectral representation 
\begin{equation} Q=\sum_{i=1}^3\lambda_1\be_i\otimes\be_i, \quad   \lambda_1+\lambda_2+\lambda_3=0. \label{Q-spectral} \end{equation}
The eigenvalues of $Q$ satisfy $ -\frac{1}{3}\leq \lambda_i\leq \frac{2}{3}$, $i=1, 2, 3$. 
For a biaxial nematic, $Q$ admits the representation in terms of the order parameters $r$ and $s$,
\begin{equation}
Q=r (\be_1\otimes\be_1-\frac{1}{3}I)+ s(\be_2\otimes\be_2-\frac{1}{3}I),  \label{Q-rs}
\end{equation} 
where 
$$s= \lambda_1-\lambda_3= 2\lambda_1+\lambda_2, \quad r= \lambda_2-\lambda_3=\lambda_1+2\lambda_2.$$
The inequality constraints on $\lambda_i$ imply restrictions on $r $ and $s$. Specifically, admissible values of  $(r,s)$ belong to
 the interior of the triangle $\mathcal T$ determined by the edges $\partial \mathcal T$: $r+s=1, r-2s=1 $ and $s-2r=1$ (Figure \ref{fig:majumdar}). 
It is easy to check that $Q$ reaches its minimum eigenvalue $\lambda=-\frac{1}{3}$ on each edge of $\partial \mathcal T$.  Hence, 
\begin{equation}
\det L=0 \Leftrightarrow \det(Q+\frac{1}{3}I)=0  \Leftrightarrow (r,s)\in\partial\mathcal T.\label{invertibility}
\end{equation}

 \begin{figure}[htbp]
 \centering
  \includegraphics[scale=0.4]{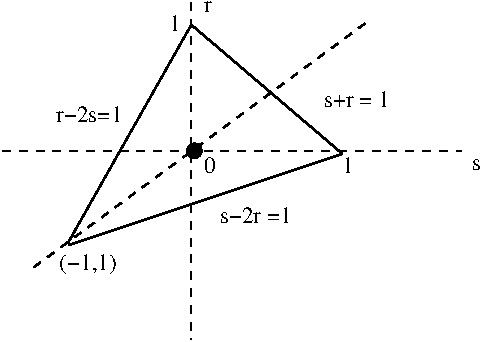}  
 \caption{$Q$ attains the minimum eigenvalue $\lambda=-\frac{1}{3}$ on each one of the edges of the triangle $\mathcal T$ \cite{majumdar2010landau}.}  \label{fig:majumdar}
 \end{figure}

We finally notice that the  uniaxial states correspond to the lines $r=0, s=0 $ and  $r=s$. In the latter case, the uniaxial order 
tensor representation is
\begin{equation}
Q= -s(\bn\otimes\bn-\frac{1}{3}I), \,\, -s\in(-\frac{1}{2}, 1), \, |\bn|=1, \label{Q}
\end{equation} 
with $\bn:=\be_3$ representing the director of the theory.
Consistency with (\ref{L-biaxial}) yields
\begin{equation}
a_0r=(l_1-l_3), \quad a_0 s=(l_2-l_3).
\end{equation}
This reduces to the uniaxial nematic with director $\bn$ and order parameter $s$
in the case that 
$$l_2=l_1:=l_\perp, \,\, l_3:=l_\|.$$

We conclude this section  introducing the following notation.
Let ${\mathbb M}^3$ denote the space of three-dimensional tensors.   
\begin{eqnarray}
&& \mathbb M^3_{+}=\{M\in\mathbb M^3: \det M>0\}, \quad
\mathbb{S}^3= \{M\in\mathbb M^3: M=M^T\},\\
&& \mathbb{S}^3_{+}=\{M\in\mathcal S^3: \det M>0\}, \quad
\mathbb{S}^3_0=  \{Q\in\mathbb{S}^3: \, \textrm{tr}\,Q=0\}.
\end{eqnarray}
 The shape tensor $L\in\mathbb{S}^3_{+}$ and  $Q\in \mathbb S^3_0$.

\subsection{Energy functionals}
We now present the energy expressions for liquid crystals and their coupling with the anisotropic elastic energy. This motivates the liquid crystal elastomer energies  
analyzed in this article. 

\subsubsection{Uniaxial liquid crystal energy}
When using the uniaxial representation (\ref{Q}) on $Q$,  the liquid crystal
energy becomes that of  the model of Ericksen 
for liquid crystals with variable degree of orientation, 
\begin{equation}
\We= k(|\nabla s|^2 + s^2 |\nabla \bn|^2)+ \nu f(s), \label {ericksen}
\end{equation} 
where $k>0$ denotes the nematic elastic constant.  
 The bulk energy 
$f$ is assumed to be parametrized by the temperature, in thermotropic liquid
crystals, and by the rod concentration, in lyotropic ones \cite{ericksen1991liquid},   \cite{kuzuu-doi1983} and \cite{kuzuu-doi1984}. High values of the rod concentration favor the nematic state, whereas dilute systems tend to be isotropic. This motivates us to consider nematic fluids with coupling of order and rod density $\rho$ by allowing the bulk energy density to depend on $\rho$ as well, that is, taking $f(\rho, s)$, as in the study of actin networks presented in section 4.  

We point out that (\ref{ericksen}) follows from the   Landau-de Gennes energy of biaxial nematics when setting the biaxial order parameter $r=0$ in (\ref{Q-rs}).

\subsubsection{Landau-de Gennes liquid crystal energy}  In its original form, the Landau-de Gennes energy density is given by 
 \begin{equation}
\wldg= \nu f(Q) + k|\nabla Q|^2 \label{dGL}
\end{equation} 
 where 
\begin{equation}
f=  a\, \tr (Q^2) -\frac{b}{3}\tr(Q^3)+ \frac{c}{4}(\tr Q^2)^2, \quad
a=\frac{\alpha}{2}(T-T_{NI}), \label{bulk-thermo}
\end{equation}
$a$, $b$, $c$  and $\alpha$ are positive, material dependent  constants, $T $
denotes absolute temperature, and $T_{NI}$ 
is the temperature of transition between the isotropic and nematic phases.   The parameterization of $f$ by $T$ has the effect of changing the relative depth of the potential wells, with the nematic minimum prevailing at low temperatures whereas the isotropic one has lowest energy at high temperature. 

The polynomial form of the bulk energy $f(Q)$ poses physical and mathematical difficulties, so instead we
assume that there exists a smooth function $\Phi: \mathcal T\to
\boldsymbol R+$ such that
\begin{eqnarray}
&&\,\,\,f(Q)= \Phi(s,r), \label{Phi1}\\
&&\lim_{(s.r)\to \partial\mathcal T}\Phi(s,r)=+\infty. \label{Phi3}
\end{eqnarray}

\subsubsection{Incompressible liquid crystal elastomer}
The total energy of a liquid crystal elastomer couples the  anisotropic elastic free energy  (\ref{btw0}) with the Landau-de Gennes liquid crystal expression is given by
 \begin{equation}
\mathcal E =
\int_{\Omega} \frac{\mu}{2}\big(|L^{-\frac{1}{2}}F L_0^{\frac{1}{2}}|^2 \big)+
\nu f(Q) + k|\nabla Q|^2\,d\XX, \label{bwt-total}
\end{equation}
with $F$ satisfying $\det F=1$. 
For an incompressible uniaxial  elastomer,  it becomes
\begin{equation}
\mathcal E=\int_{\Omega} \big(\wbtw(F, s,\bn, \bn_0)+ \We(s, \bn)\big)\,d\XX,
\label{total-energy-incomp}
\end{equation}
with
\begin{eqnarray}
\wbtw=&& \frac{1}{2} \mu
\frac{l_{\perp}^0}{l_{\perp}}\big((|F|^2-(1-r)|F^T\bn|^2)\nonumber\\
+ &&(\frac{1-r_0}{r_0})(|F\bn_0|^2-(1-r)(F^T\bn\cdot\bn_0)^2)\big),\nonumber\\
r=&& \frac{l_\perp}{l_\|}, \,\, r_0=  \frac{l_\perp^0}{l_\|^0}, \quad \alpha:=
1-r,\label{btw}
\end{eqnarray}
as in (\cite{warner2007liquid}), with $\bn_0, l_{\|}^0$ and $l_{\perp}^0$ representing rod alignment and polymer shape, respectively, at crosslinking. From (\ref{btw}), we observe that the configuration at crosslinking  is also the reference one.  

\subsubsection{Compressible liquid crystal elastomer}
The energy of a compressible liquid crystal elastomer follows from  (\ref{bwt-total})  but now allowing deformations  such that $\det F>0$, without the constraint of the determinant being equal to 1.  Moreover, as in isotropic nonlinear elasticity,  we assume that $f$ is also a
function of $\det F$. That is, 
 \begin{equation}
\mathcal E =
\int_{\Omega} \frac{\mu}{2}\big(|L^{-\frac{1}{2}}F L_0^{\frac{1}{2}}|^2 \big)+
\nu f(Q, \det F) + k|\nabla Q|^2\,d\XX, \label{energy-compressible-total}
\end{equation}
where $\mu\geq 0, \nu\geq 0$ and $k\geq 0$ are prescribed material parameters. 
We  point out that the function $f$ also encodes the isotropic to nematic phase transition behavior, and, in particular, it accounts for the observed change
 in volume in such a transition.

\subsubsection{Nonlinear anisotropic elastic energy}
We observe from the previous subsections that, due to the anisotropy,   the energy of the liquid crystal elastomer  depends on the deformation gradient $F$ through the combination 
\begin{equation}
G:= L^{-\frac{1}{2}}FL_0^{\frac{1}{2}}= (L^{-\frac{1}{2}}V)(RL_0^{\frac{1}{2}}):=
\tilde V\tilde R,   \label{G}\end{equation}
Note that, as result of the polar decomposition of $F=VR$, where $V$ is a symmetric positive definite tensor and $R$ is a proper rotation,  the deformation  expressed by $G$ can be viewed as a composition of the rotation of the elastic network in the reference configuration and an anisotropic stretch. 

Hence, we  consider stored energy functions of liquid crystal elastomer of the form
\begin{equation}
W(\bX)= W(L(\bX), L_0(\bX), F(\bX))= \hat W(G(\bX)). \label{hatW}
\end{equation}
This setting allows us to explore the analytic tools of isotropic elasticity, but with the significant difference that now, the tensor $G$ is not a gradient. 
 Let us consider the elastomer energy
\begin{equation}
\mathcal E= \int_{\Omega} \hat W(G) + |\nabla Q|^2+ f(Q)- \mathcal
L(\boldsymbol \bvphi), \label{dgl-elast}
\end{equation}
allowing the form $f(Q, \det F)$ in the compressible case. Note that the last term is a linear functional of the gradient map $\bvphi$ and   corresponds to subtracting an externally supplied mechanical energy. Since its treatment follows that of isotropic elasticity \cite{CI87}, from now on, we do not include it in the total energy. 

\section{Energy Minimization} 
We start with  making the following assumptions on $\hat W(G)$ motivated by the analogous ones in isotropic nonlinear elasticity 
\cite{ball1976convexity}.
\smallskip

\noindent 
{\it Polyconvexity}: there exists a convex function $\Psi:
\mathbb{M}^3_{+}\times \mathbb{M}^3_{+}\times (0,\infty) \longrightarrow
\mathbb{R}$ 
such that  $\hat W$ in 
(\ref{hatW}) satisfies
 \begin{equation}
\hat W(G)= \Psi(G, \adj G, \det G). \label{polyconvexity}
 \end{equation}
 
\noindent 
{\it Coerciveness}: There exist constants $\alpha, \beta, p , q, r$ such that 
\begin{eqnarray}
&& \alpha>0, \, p\geq 2, q\geq {\frac{p}{p-1}}, r>1, \nonumber\\
&& \hat W(G)\geq \alpha(|G|^p+ |\adj G|^q +(\det G)^r)+\beta, \nonumber \\
&&\textrm { for almost all } \XX\in \Omega \textrm{ and for all } F\in
\mathbb{M}^3_{+}.
\label{coerciveness}\end{eqnarray}

\noindent
{\it Growth near zero-determinant}:
\begin{equation}
\lim_{\det G\to 0^+} \hat W(G)=+\infty. \label{detgrowth}
\end{equation}

\noindent
{\bf Remark. \,} We observe that the condition on the
exponents $p$ and $q$  of (\ref{coerciveness})  guarantees 
  $\frac{1}{p}+\frac{1}{q}<\frac{4}{3}$. This is a 
required condition to obtain convergence of  weak limits 
of sequences of determinants, in the  proof of existence of minimizer (Theorem \ref{determinant-limits}). 

\subsection{Auxiliary results}
We now present some auxiliary results needed in the proof of existence of minimizer. We would like to point out that, in our literature 
review we didn't found their corresponding proofs. 
\begin{proposition} \label{prop:FTLFest}
 Let  $L \in \mathbb{S}^3_{+}$. Then for any matrix $F
\in \mathbb{M}^3$
 \begin{equation}
  \tr (F^T L F) \geq l_{\min}(L) |F|^2, 
 \end{equation}
 where $l_{\min}(L)>0$ is the smallest eigenvalue of $L$.
\end{proposition}
\begin{proof}
 Let us consider the spectral decomposition of  $L=S^T D S$, where $S$ is an orthogonal matrix and  $D=\mathrm{diag}(d_1, d_2, d_3)$ with $0 <
d_1 \leq d_2 \leq d_3$. 
Let us denote $A = SF (SF)^T$ and calculate
 \begin{eqnarray*}
  \tr (F^T L F) 
  &=& \tr (F^T S^T D S F) 
  = \tr (AD)   
  = \sum_{i=1}^3 d_i A_{ii} \\
  &\geq& d_1 \sum_i A_{ii} 
  = d_1 \tr (F F^T) 
 = l_{\min}(L) |F|^2,  
 \end{eqnarray*}
 where we have used the fact that $A_{ii} \geq 0,$ for $i=1,2,3$.
\end{proof}
\begin{proposition} \label{prop:detovernorm}
 Assume $L \in \mathbb{S}^3_{+}$, then we have
 \begin{equation}
  \frac{\mathrm{det}(L)}{|L|} \geq \frac{1}{\sqrt{3}} l_{\min}^2(L).
 \end{equation}
\end{proposition}
\begin{proof}
Denoting the eigenvalues of $L$, $0< l_1 \leq l_2 \leq l_3$,
we have
 \begin{eqnarray*}
  |L| 
  &=& \sqrt{\tr (L^T L)} \\
  &=& \sqrt{l_1^2 + l_2^2 + l_3^2}. 
 \end{eqnarray*}
 Hence
 \begin{eqnarray*}
  \frac{\mathrm{det}(L)}{|L|}
  &=& \frac{l_1 l_2 l_3 }{\sqrt{l_1^2 +l_2^2 +l_3^2}} 
  \geq \frac{l_1 l_2 l_3}{\sqrt{3} l_3} \\
  &=& \frac{1}{\sqrt{3}} l_1 l_2 
  \geq \frac{1}{\sqrt{3}} l_1^2 
  = \frac{1}{\sqrt{3}} l_{\min}^2(L).
 \end{eqnarray*}
\end{proof}
\begin{lemma} \label{lemma:G-adjG}
For given $L, L_0\in \mathbb{S}^3_{+}$,  let $G$ be as in (\ref{G}). Then the following inequalities hold, 
\begin{eqnarray}
&&|G|\geq C_1 |F|, \,\, \textrm {\, and} \label{coercivityG}\\
&& |\adj G|\geq C_2 |\adj F|,  \label{coercivityadjG}
\end{eqnarray}
where $C_1=\sqrt{\frac{l_{min}(L_0)}{l_{max}(L)}}$ and $C_2= \frac{1}{3} C_1^2.$
\end{lemma}

\begin{proof}
 We have
\begin{equation}
 |G|^2 = \tr \left(L_0 F^T L^{-1} F \right).
\end{equation}
 Since $L_0 \in \mathbb{S}^3_{+}$, we let $K_0=\sqrt{L_0}$. 
 Applying Proposition \ref{prop:FTLFest}, we estimate
 \begin{eqnarray*}
  |G|^2 &=& \tr \left(K_0^T F^T L^{-1} F K_0 \right) \\
  &\geq& l_{\min}(L^{-1}) |F K_0|^2 
=l_{\min}(L^{-1}) \tr (F L_0 F^T) \\
  &\geq& l_{\min}(L^{-1}) l_{\min}(L_0) |F|^2 = \frac{l_{\min}(L_0)}{l_{\max}(L)} |F|^2.
 \end{eqnarray*}
This yields
 \begin{equation}
  |G| \geq \sqrt{\frac{l_{\min}(L_0)}{l_{\max}(L)}} |F|.
 \end{equation} To prove \ref{coercivityadjG}, we calculate

\begin{eqnarray*}
 \mathrm{adj}(G)
 &=& \mathrm{det}(G) G^{-1} \\
 &=& \mathrm{det}\left(L^{-1/2} F L_0^{1/2} \right) L_0^{-1/2} F^{-1} L^{1/2} \\
 &=& \mathrm{det}\left(L^{-1/2}\right) \mathrm{det}\left(L_0^{1/2}\right) 
L_0^{-1/2} \mathrm{adj}(F) L^{1/2}.
\end{eqnarray*}
So
\begin{equation}
 L_0^{1/2} \mathrm{adj}(G) L^{-1/2}
 = \mathrm{det}\left(L^{-1/2}\right) \mathrm{det}\left(L_0^{1/2}\right)
\mathrm{adj}(F).
\end{equation}
By the matrix  property $|A B| \leq |A| |B|, $  
we have 
\begin{equation} \label{eqn:adjG-est1}
 |L_0^{1/2}|\cdot |\mathrm{adj}(G)|\cdot |L^{-1/2}| 
 \geq \mathrm{det}\left(L^{-1/2}\right) \mathrm{det}\left(L_0^{1/2}\right)
|\mathrm{adj}(F)|.
\end{equation}
Using Proposition \ref{prop:detovernorm}, we have
\begin{eqnarray*}
 |\mathrm{adj}(G)| 
 &\geq& \frac{1}{\sqrt{3}} l_{\min}^2(L^{-1/2}) \frac{1}{\sqrt{3}}
 l_{\min}^2(L_0^{1/2}) |\mathrm{adj}(F)| \\
 &=& \frac{1}{3} \frac{l_{\min}(L_0)}{l_{\max}(L)}
|\mathrm{adj}(F)|.
\end{eqnarray*}
\end{proof}

The following theorem is a special case of  (\cite{ball1976convexity},
Theorem 6.2). 
We will  apply it in the proofs of the existence of minimizer of the total
energy as in the case of  isotropic elasticity.  
\begin{theorem}  \label{determinant-limits} Suppose that $\Omega\subset \mathbb{R}^n$ is open. 
\begin{itemize}
\item $n=3$:
If $ u_r\wconv u$ in $W^{1,p}(\Omega)$ and $\adj \nabla u_r\wconv \adj\nabla u$
in $L^q$ with $p>1$, $q>1$ and $\frac{1}{p}+\frac{1}{q}>\frac{4}{3}$, then $\det
\nabla u_r\longrightarrow \det \nabla u$ in $\mathcal D'(\Omega)$.  
\item $n=2$:
If  $u_r\wconv u$ in $W^{1,p}(\Omega)$, and $p>\frac{3}{2}$ then $\det \nabla
u_r\longrightarrow\det \nabla u$ in $\mathcal D'(\Omega)$. 
\end{itemize}
\end{theorem}

\subsection{Incompressible  Landau-de Gennes Elastomer}
In this section, we prove existence of energy minimizer for two different forms of the liquid crystal bulk energy $f(Q)$.  
We first consider the case that $f(Q)$ is the standard Landau-de Gennes polynomial and in our second approach, we assume that the 
function $f(Q)$ is defined as in \ref{Phi1} and \ref{Phi3}.

\subsubsection{Restriction on the domain of $Q$}

We now include a new constraint on the elements $Q\in\mathbb S_0^3$ of the admissible set:   for a given $\varepsilon>0$,  $\lambda_{\min}(Q)
\geq -\frac{1}{3}+\varepsilon$. 

We note that even 
 the strict bound  $\lambda_{\min}(Q) > -\frac{1}{3}$  is not sufficient to ensure that the limit $Q^*$
  of  the minimizing sequences $\{Q_k\}$ satisfies the same strict lower bound so as to guarantee the invertibility of $L$ obtained from (\ref{QL}). 

\noindent
As in \cite{cesana2008strain}, we define the set
\begin{equation}
 \mathcal{Q}(a) = \{Q \in \mathbb{S}^3_0, \, \lambda_{\min}(Q) \geq a\}, 
\end{equation}
where $a$ is some real number, and define
$\mathcal{Q}_{\varepsilon}=\mathcal{Q}(-1/3+\varepsilon)$, 
where $0<\varepsilon \leq 1/3$ is an arbitrary constant.

\begin{proposition} \label{prop:Qa-convex}
 The set $\mathcal{Q}(a)$ is convex in $\mathbb S_0^{3}$.
\end{proposition}
\begin{proof}
 Take any two matrix $Q_1$ and $Q_2$ in $\mathcal{Q}(a)$, and let
 \begin{equation*}
  Q = \alpha_1 Q_1 + \alpha_2 Q_2,
 \end{equation*}
 where $\alpha_i \geq 0, i=1, 2$, and $\alpha_1+\alpha_2=1$.
 Since  $Q\in \mathbb{S}_0^3$,  we only need to show
that $\lambda_{\min}(Q) \geq a$.
 By Rayleigh's formula, we have
 \begin{eqnarray*}
  \lambda_{\min}(Q) 
  &=& \min_{|\bx|=1} \bx^T Q \bx 
  = \min_{|\bx|=1}\left( \sum_{i=1}^2 \alpha_i \bx^T Q_i \bx \right) \\
  &\geq& \sum_{i=1}^2 \alpha_i \min_{|\bx|=1} \bx^T Q_i \bx 
  = a
 \end{eqnarray*}
Hence $Q\in\mathcal{Q}(a)$  and so, the convexity of $\mathcal{Q}(a)$ follows. 
\end{proof}

 For any matrix $Q \in \mathbb{S}^3_0$, since $\tr (Q)=0$, 
  $\lambda_{\min}(Q) \leq 0$ and $\lambda_{\max}(Q) \geq 0$ hold. 
 The following proposition gives a bound of $\lambda_{\max}(Q)$ based 
 on $\lambda_{\min}(Q)$.
\begin{proposition} \label{prop:lammax-lammin}
Let  $Q\in \mathbb{S}^3_0$. Then
 \begin{equation} \label{eqn:lammax-lammin}
  \lambda_{\max}(Q) \leq -2 \lambda_{\min}(Q).
 \end{equation}
\end{proposition}
\begin{proof}
 For any matrix $Q$ in $\mathbb{S}^3_0$,  let  its eigenvalues satisfy 
$  \lambda_1 \leq \lambda_2 \leq \lambda_3.$
 Since $\tr (Q)=\lambda_1+\lambda_2+\lambda_3 =0$, we have
 \begin{eqnarray*}
  -\lambda_3 &=& \lambda_1 + \lambda_2 
  \geq 2 \lambda_1.
 \end{eqnarray*}
The conclusion follows by multiplying both sides of the previous inequality  by $-1$. 
\end{proof}

Now we turn to the questions of estimating eigenvalues of $L$ and $L_0$ given by (\ref{QL-1}) for $Q, Q_0\in\mathcal Q_\epsilon$.
Note that $L$ and $L_0$ are both symmetric and positive definite, so that the anisotropic deformation tensor $G$  in (\ref{G}) is well defined. 
By Proposition \ref{prop:lammax-lammin}, we have that
\begin{eqnarray*}
 \lambda_{\max}(Q) 
 &\leq& -2 \lambda_{\min}(Q) 
 \leq \frac{2}{3}-2\varepsilon.
\end{eqnarray*}
So, using the constitutive relation (\ref{QL-1}) gives
\begin{equation}
 \lambda_{\max}(L) \leq a_0(1-2\varepsilon) \leq a_0.
\end{equation}
Moreover,  since $\lambda_{\min}(Q_0) \geq -\frac{1}{3}+\varepsilon$, 
using equation (\ref{QL-1}) again yields
\begin{equation}
 \lambda_{\min} (L_0) \geq a_0 \varepsilon.
\end{equation}
Hence, from Lemma \ref{lemma:G-adjG}, we have that
\begin{eqnarray}
&& |G| \geq \sqrt{\varepsilon} |F|, \quad \textrm{ and} \label{eqn:est-of-G}\\
&& |\mathrm{adj}(G)| \geq \frac{1}{3} \varepsilon |\mathrm{adj}(F)|. \label{eqn:est-of-adjG}
\end{eqnarray}
We consider the problem of minimizing  (\ref{dgl-elast}) on the admissible set
\begin{eqnarray}
\mathcal A_{\varepsilon}=&&\{\boldsymbol \varphi\in W^{1,p}(\Omega,
\mathbb{R}^3),
\, Q\in W^{1,2}(\Omega,\mathcal{Q}_{\varepsilon}): \, \, \adj(\nabla\boldsymbol\varphi)\in L^q(\Omega, \mathbb M^3), \,\, \nonumber\\
&& \quad \det\nabla\bvphi=1, \bvphi=\hat{\bvphi} \textrm{ and } Q=\hat Q \textrm
{ on } \Gamma_0\}, \label {Admissible-general}
\end{eqnarray}
where $p$ and $q$ are as in (\ref{coerciveness}).
The following theorem proves existence of a global minimizer of the energy.
\begin{theorem} \label{thm:dgl-incomp}
Let $\Omega\in \mathbb{R}^3 $ be open and bounded, and with smooth boundary
$\partial\Omega$. Let $\Gamma_0, \Gamma\subset \partial\Omega$, 
with $\Gamma_0\cap\Gamma=\emptyset$.  Let $W:
\mathbb{S}^3_{+}\times\mathbb{S}^3_{+}\times
\mathbb{M}^3_{+}\longrightarrow \mathbb{R}$ be as in (\ref{hatW}) satisfying 
hypotheses (\ref{polyconvexity}) and  (\ref{coerciveness}).
Then, there exists at least one pair $(\bvphi^*, Q^*)\in \mathcal
A_{\varepsilon}$ such that 
\begin{equation} 
\mathcal E(\bvphi^*, Q^*)=\inf_{(\bvphi, Q)\in \mathcal A_{\varepsilon}}
\mathcal E(\bvphi, Q).
\end{equation}
\end{theorem}
\begin{proof}
First of all, we point out that the integrals in the definition of $\mathcal E$
are well defined. We observe as well that
$\mathcal A_{\varepsilon}\neq\emptyset$ and therefore, there exists a constant
$K_1>0$ such that  the following inequality holds:
\begin{equation}
\inf_{(\bvphi, Q)\in \mathcal A_{\varepsilon}} \mathcal{E} < K_1.
\label{infupperbound}
\end{equation}
\noindent {\bf Step 1: \,Coercivity.}
From  the coercivity hypothesis (\ref{coerciveness})  on $\hat W(G)$ and the
form of the Landau-de Gennes energy, it follows that 
\begin{eqnarray}
\mathcal E(G, Q)
 &\geq& \alpha\int_{\Omega}\big(|G|^p+ |\adj G|^q+ |\nabla Q|^2\big)\,d\XX. \label{total-energy-coercivity}
\end{eqnarray}
Likewise, the positivity of $\hat W(G)$ implies that 
\begin{equation}
\mathcal E(G,Q)\geq\int_{\Omega}f(Q)\,d\bX. \label{bound-f-integral}
\end{equation}
According to the generalized Poincar$\mathrm{\acute{e}}$ inequality
(\cite{CI87}, p281), there exists a constant $c>0$ such that 
\begin{equation}
\int_{\Omega}|\bvphi|^p\,d\XX\leq c\{\int_{\Omega}|\nabla\bvphi|^p \,d\XX+
|\int_{\Gamma_0}\bvphi \,dS|^p\}, \label{poincare}
\end{equation}
for all $\bvphi\in W^{1,p}(\Omega)$.  
Likewise,
\begin{equation}
\int_{\Omega}|Q|^2\,d\XX\leq c\{\int_{\Omega}|\nabla Q|^2 \,d\XX+
|\int_{\Gamma_0}Q\,dS|^2\}. \label{poincareQ}
\end{equation}
Now,  combining  (\ref{eqn:est-of-G}), (\ref{eqn:est-of-adjG}), (\ref{poincare})
and (\ref{poincareQ}), 
with the fact that $p\geq 2$ gives the  existence of  constants $C>0$ and $ c_0$
such that 
\begin{equation}
\mathcal E(\bvphi, Q)\geq C \| \bvphi\|^p_{1,p} + \|\adj
F\|^q_{0,q}+\|Q\|^2_{1,2}- c_0. \label{total-coercivity}
\end{equation} 
The latter inequality  guarantees the existence of a constant $K_0$, which
together with (\ref{infupperbound}) yields
\begin{equation}
K_0<\inf_{(\bvphi, Q)\in \mathcal A_{\varepsilon}}\mathcal E< K_1.
\label{infbounds}
\end{equation}
Let $(\bvphi_k,  Q_k)\in\mathcal A_{\varepsilon} $  be a minimizing sequence for 
$\mathcal E$, that is 
\begin{equation}\lim_{k\to\infty} \mathcal E(\boldsymbol \varphi_{\varepsilon}, Q_{\varepsilon})= \inf_{(\bvphi, Q)\in\mathcal
A_{\varepsilon}}\,\mathcal E.\end{equation}
\noindent{\bf Step 2: Compactness.}
From inequality  (\ref{total-coercivity}), it follows that 
$$\mathcal E(\bvphi_k, Q_k)\longrightarrow \infty, \textrm {as }
(\|\bvphi_k\|_{1,p}+\|\adj F_k\|_{0,q}+\|Q_k\|_{1,2})\to \infty,$$ 
which together with the second inequality in  (\ref{infbounds}) 
imply that 
$$   (\bvphi_k, \adj \nabla\bvphi_k, Q_k) \,\, \textrm { is bounded in the reflexive Banach
space}\,\,  W^{1,p}\times L^q\times W^{1,2}.$$
Therefore there exist weakly convergent subsequences such that 
\begin{eqnarray}
&&\bvphi_k\wconv \bvphi^* \,\, \textrm {in}\,\, W^{1,p}, \label{weak1}\\
&&\adj \nabla\bvphi_k\wconv H^*\,\, \textrm {in}\,\, L^q, \label{weak2}\\
&&Q_k\wconv Q^* \,\, \textrm{in }\,\, W^{1,2}, \label{weak3}
\end{eqnarray}
\noindent
{\bf Step 3: Properties of $\bvphi^*$ and $ Q^*$. }
From (\ref{weak1}) and (\ref{weak2}), we have by Theorem \ref{determinant-limits}  that
\begin{eqnarray}
 &&H^* = \adj(\nabla \bvphi^*), \quad {\textrm{and}}\nonumber\\
&&\mathrm{det}(\nabla \bvphi^*) = \mathrm{det}(\nabla \bvphi_k) = 1 \quad \text{
a.e. in }\Omega. \label{det-constraint}
\end{eqnarray}
Also, by Proposition \ref{prop:Qa-convex} and Mazur's Theorem, 
the set $\{Q\in H^1(\Omega,\mathbb{M}^3): Q \in \mathcal{Q}_{\varepsilon} \text{
a.e. in }\Omega \}$ is weakly closed.
Thus it follows from (\ref{weak3}) that $Q^* \in \mathcal{Q}_{\varepsilon}$ a.e.
in $\Omega$. 
Hence $(\bvphi^*, Q^*)\in \mathcal A_{\varepsilon}$. 

 Finally, the existence of minimizer follows from the lower semi-continuity of
$\mathcal E$, 
due to the polyconvexity assumption on $\hat W(G)$, and  the continuity of
$f$. 
This concludes the proof of the theorem. 
\end{proof}

\begin{remark}
Note that the previous result applies to the Bladon-Terentjev-Warner energy only
in the case $n=2$.  
This is a direct consequence of Theorem \ref{determinant-limits}.
\end{remark}

\subsubsection{Non-polynomial growth of the bulk energy $f(Q)$}

We now present the case that $f(Q)$ is not a polynomial and its order parameter representation satisfies the  growth conditions \ref{Phi1}-\ref{Phi3}.
Now, let 
\begin{eqnarray}
\mathcal A_{0}=&&\{\boldsymbol \varphi\in W^{1,p}(\Omega,
\mathbb{R}^3), 
\, Q\in W^{1,2}(\Omega,\mathcal{Q}_{0}): \adj(\nabla\boldsymbol\varphi)\in L^q(\Omega, \mathbb M^3), \nonumber\\
&& \quad \det\nabla\bvphi=1, \bvphi=\hat{\bvphi} \textrm{ and } Q=\hat Q \textrm
{ on } \Gamma_0\}. \label {Admissible-general0}
\end{eqnarray}

 \begin{theorem}
Let the hypotheses Theorem \ref{thm:dgl-incomp} hold, and suppose that  $f(Q)$ and $\Phi(s,r)$ satisfy the
assumptions [\ref{Phi1}]-[\ref{Phi3}].  Suppose that the boundary tensor satisfies $\lambda_{{{\textrm{\tiny min}}}}(\hat Q)>-\frac{1}{3}+\varepsilon$,
 for some $\varepsilon>0$. Then the total energy  has a
global minimizer in \ref{Admissible-general0}.

 \end{theorem}
 \begin{proof}
 It is easy to see that Step 1 and Step 2 of the proof of the previous theorem follow as well in this case. Let $\{(\bvphi_k, Q_k)\}_{k\geq 1}$ denote a minimizing sequence of the energy in $\mathcal A_0$.    
 \smallskip
 
 \noindent {\bf Step 3: properties of $\bvphi^*$ and  $Q^*$.} First of all, we note that (\ref{det-constraint}) also holds in this case.
  We now study properties of the minimizing sequence $\{Q_k\}$ to show that $Q^*\in\mathcal Q_0$.  We first observe that the strong-convergence of $\{Q_k\}$ to $Q^*$ in $L^2$ follows from (\ref{weak3}), and,  up to a subsequence,  it implies that
 \begin{equation}
 Q_k\longrightarrow Q^* \textrm{\, a.e.  in\, } \Omega.
 \end{equation}
 Let  $q^m:=\det (Q_m+ \frac{1}{3}I),$   and note that 
 $$q^m>0 \, \textrm{a.e. \,} \Omega \Leftrightarrow q^*\geq 0 \,  \textrm{a.e. \,} \Omega.$$ 
 We want to prove that $q^*>0$ a.e. $\Omega$. For this, supposes that $q^*=0$ on a set $A\subset\Omega$, ${\textrm{vol}}(A)>0.$   Since $0<d^l\wconv d^*$, we have
 $$\int_A|\det(Q^l+\frac{1}{3}I)|\,d\XX = \int_A\det(Q^l+\frac{1}{3}I)\,d\XX\longrightarrow \int_A\det(Q^*+\frac{1}{3}I)\,d\XX=0.$$
 We now consider the sequence $f^m:=f(Q^m)$ of measurable functions of $\XX$. Since $f^m\geq 0$, by Fatou's theorem
 $$\int_A \liminf_{m\to\infty} f^m(\XX)\,d\XX\leq  \liminf_{m\to\infty}\int_A f^m(\XX)\,d\XX.$$
 By the growth assumption (\ref{Phi3}) on $f$
  $$ \liminf_{m\to\infty} f^m=  \lim_{\det(Q+\frac{1}{3}I)\to 0} f(Q)=+\infty,$$
 and consequently $ \lim_{m\to\infty}\int_A f(Q^m(x))\,d\XX= +\infty.$
  But the latter relation, contradicts the statement that $\int_{\Omega} f(Q_k)< K_1$ that follows from
  (\ref{bound-f-integral}).
  Hence $Q^*\in\mathcal Q_0 $ a.e $\Omega$. 
  
  Finally, existence of energy minimizer in $\mathcal A_0$ follows from the polyconvexity of $\hat W$, the weak lower semicontinuity of $\int_{\Omega} f$ that follows from Fatou's theorem, and the fact that the pair $(\bvphi^*, Q^*)$ satisfies the boundary conditions prescribed to the elements of $\mathcal A_0$.  The latter is a consequence of the compactness of the trace operator mapping $W^{1,p}(\Omega)$ to $L^p(\Omega)$ (and the analogous one for the tensor $Q$).
\end{proof}

\subsection{Compressible Landau-de Gennes Elastomer}
In this section we analyze energy minimization in the case that the elastomer is compressible.   This brings the new feature of  the coupling between changes of volume of the network and nematic order. We will first propose and analyze energy expressions accounting for the new property, and focus in the case of a rod fluid with elastic couplings. 

We propose a free energy density of the form
\begin{eqnarray}
&&\Psi_{tot}(G, L, Q, F)= \hat W(G) + f_{\textrm{\begin{tiny}B\end{tiny}}}(\det F, Q), \label{Psiform} \\
&& \fb(\det F, Q)= f(\det F, Q) + g(\det F), \label{fb}
\end{eqnarray} 
where 
$f: (0, \infty)\times \mathcal S_0^3\longrightarrow [0, \infty)$ and $g:(0,\infty)\longrightarrow [0, \infty)$ are prescribed continuous functions. 
We assume that 
\begin{eqnarray}
&&\lim_{\det F\to (0, \infty)} f(\det F, Q)= +\infty, \quad {\textrm{for each }} Q\in \mathcal S_0^3, \label{fdetgrowth}\\
&&\lim_{\det (Q+\frac{1}{3}I)\to 0} f(\det F, Q)= +\infty, \quad {\textrm{for each }} F\in \mathcal M^3_+, \label{fQgrowth}
\end{eqnarray}
The latter growth conditions are also postulated in \cite{ballmajumdar2010} in studying the compatibility of the Landau-de Gennes theory of nematic with the mean filed theory of Maier and Saupe.

Let  $\hat{\bvphi}\in H^{\frac{1}{2}}(\Gamma_0, \mathbb{R}^3)$ and  $\hat Q\in
H^{\frac{1}{2}}(\Gamma_0, \mathbb{S}^3_0)$ be prescribed. 
The admissible set is
\begin{eqnarray}
&&\mathcal {AC}=\{\boldsymbol \varphi\in W^{1,p}(\Omega,\mathbb{R}^3), 
\, Q\in W^{1,2}(\Omega,\mathcal{Q}_{0}): \adj G\in L^q(\Omega, \mathbb M^3), \nonumber\\ &&\det G\in L^r(\Omega, R^+), 
 \, \det\nabla\bvphi>0, \bvphi=\hat{\bvphi} \textrm{ and } Q=\hat Q \textrm
{ on } \Gamma_0\}. \label{Admissible-general-set}
\end{eqnarray}
We now establish the following existence theorem.
\begin{theorem} Suppose that the free energy density is as in (\ref{Psiform}), with the bulk contribution given by (\ref{fb}). Suppose that the assumptions (\ref{polyconvexity}), (\ref{coerciveness}), (\ref{detgrowth}),  (\ref{fdetgrowth}) and (\ref{fQgrowth})  hold. We additionally require that one of the following holds:
\begin{enumerate}
\item If $f$ in (\ref{fb}) is convex with respect to $\det F$,  then  we set $g\equiv 0$.
\item  If $f$ is nonconvex with respect to $\det F$, then $g\neq 0$ is smooth and convex. 
\end{enumerate}
Then the total energy admits a minimizer in $\mathcal{AC}$. 
\end{theorem}
\begin{proof}
We observe that Step 1 and Step 2 of the proof of Theorem \ref{thm:dgl-incomp} apply to this case as well. We need to establish that $(\bvphi^*, Q^*)$ belong to the admissible set, by showing that   
 $\det \nabla\bvphi^*>0$. If  $g\neq 0$, it follows by Fatou's theorem along the same lines of the proof of $\det\nabla\bvphi^*=1$ in the incompressible case. 
 If $g\equiv 0$, the proof can also be given using Mazur's theorem on $f$. Existence of minimizer follows from the polyconvexity of $\hat W$ together with Fatou's theorem that provides the weak lower semicontinuity of $\fb$.
\end{proof}
\begin{remark}
The energy minimizer may not be uniaxial even in the case that $L_0$ and  $\hat
Q$  are uniaxial. 
In fact, the same statement is true for the minimizer of the Landau-de Gennes
energy of   the pure liquid crystal problem
(\cite{majumdar2010landau}).
In that case, numerical results give strong evidence of uniaxiality when $\hat
Q$ is uniaxial.  
\end{remark}
\begin{remark}
The two sets of assumptions on $\fb$ are meant to deal with the convexity properties of  $f(\det F, Q)$ with respect to $\det F$. 
Nonconvexity will occur in cases where phase transitions are involved, in which case $g$ will act to control the determinant. 
On the other hand, the convexity of $f$ it is sufficient to provide information on the weak limits of sequences of determinants. 
\end{remark}

\subsubsection{Rod fluids with elastic crosslinks} 
We now propose a model for the elastically interacting nematic units motivated by the models of actin networks. The free energy density  is now of the form
\begin{equation}
\Psi_{tot}(\rho, G, L, Q)= \hat W(G) + k|\nabla Q|^2+g(\det F)+  f(\rho, Q) + \varepsilon |\nabla \rho|^2, \label{Psiform-rho}
\end{equation}
with $\hat W$ as in (\ref{Psiform}), and $\epsilon>0$ a prescribed constant. We assume that 
\begin{eqnarray}
&&f\in C((0, \infty)\times \mathcal S_0^3,  R^+), \label{frho-conts}\\
&& g \geq 0 \, {\textrm{ is smooth and convex}}, \label{g-conv}\\ 
&& \lim_{\rho\to \{0, \infty\}} f(\rho, Q)=+\infty, \,\, { \textrm{for each } } \, Q\in S^3_0, \label{frho-lim1} \\
&&  \lim_{\det(Q+\frac{1}{3}I)\to 0} f(\rho, Q)=+\infty, \,\, { \textrm{for each } } \, \rho \in (0, \infty). \label{frho-lim2}
\end{eqnarray}
We  let $\rho_0>0$ denote  the prescribed reference rod density and  assume that the equation of balance of mass
\begin{equation}
\rho(\XX) \det(\nabla\bvphi)(\XX)=\rho_0, \,\, \XX\in\Omega\, \textrm{a.e.} \label{mass-balance}
\end{equation}
of rods holds. We define the admissible set as
\begin{eqnarray}
\mathcal A_\rho=&&\{\boldsymbol \varphi\in W^{1,p}(\Omega,\mathbb{R}^3), 
\, Q\in W^{1,2}(\Omega,\mathcal{Q}_{0}): \adj G\in L^q(\Omega, \mathbb M^3), \nonumber\\ &&\det G\in L^r(\Omega, R^+), 
\, \rho\in W^{1,2}(0,\infty), 
 \, \rho>0 \, \textrm { a.e. } \Omega,\, (\ref{mass-balance}) \textrm { holds}, \nonumber\\
 \quad && \rho=\hat \rho, \, \bvphi=\hat{\bvphi} \textrm{ and } Q=\hat Q \textrm
{ on } \Gamma_0\}. \label {Admissible-rho}
\end{eqnarray}
 We now formulate the main theorem of existence of minimizer for the rod system.
\begin{theorem} \label{thm:5} Suppose that the free energy density is as in (\ref{Psiform-rho}) with $g\equiv 0$.  Suppose that the assumptions (\ref{polyconvexity}), (\ref{coerciveness}), (\ref{detgrowth}),
(\ref{frho-conts}), (\ref{frho-lim1}) and (\ref{frho-lim2})  are satisfied.  
Then the total energy admits a minimizer in $\mathcal{A_{\rho}}$. 
\end{theorem}
\begin{proof}
Once more, steps 1 and 2 of the proof   follow as in Theorem \ref{thm:dgl-incomp}. 
This yields sequences $\{\phi_k, Q_k\}$ with properties (\ref{weak1}), (\ref{weak2}),  (\ref{weak3}). Moreover by  Theorem \ref{determinant-limits}, we have that  $\det\nabla\bvphi_k \to \det\nabla\bvphi^*>0$ in  $D'(\Omega)$. 

Let $\{\rho_k\}$ denote the  minimizing sequence corresponding to the rod density.  It is easy to see that  the analog of  relations (\ref{infbounds}) and (\ref{total-coercivity}) now yield
\begin{equation} \int_{\Omega}|\nabla\rho_k|^2\,d\XX < K.   \end{equation}
This, together with Poincar{\'e}'s inequality yields a subsequence such that $\rho_k\to \rho^*$ in $L^2$ and $\rho_k \to \rho^*$ a.e.  $\Omega$. The a.e. convergence of 
$\{\rho_k\}$ to $\rho^*$, and the convergence of the sequence of determinants  yields 
$$\rho_0=\rho_k\det\nabla\bvphi_k\wconv \rho^*\det\nabla\bvphi^*.$$
Hence the balance of mass equation (\ref{mass-balance}) is satisfied at the limit. The verification that the limiting fields satisfy the boundary conditions in 
(\ref{Admissible-rho}) follow as in Theorem \ref{thm:dgl-incomp}. 
\end{proof}

Finally, we state the theorem of existence of minimizer in the case that $\Phi$ in (\ref{Psiform-rho}) is  independent of $\adj F$ and $\det F$, that is, the elastomer free energy density function is that of a Hadamard material \cite{CI87}.
In this case, we set the admissible set as 
\begin{eqnarray}
\mathcal A_{\textrm{\begin{tiny}H\end{tiny}}}=&&\{\boldsymbol \varphi\in W^{1,p}(\Omega,\mathbb{R}^3), 
\, Q\in W^{1,2}(\Omega,\mathcal{Q}_{0}), \, \rho\in W^{1,2}(0,\infty), \nonumber\\
 \,\, && \rho>0 \, \textrm {a.e.}\, \Omega,\, (\ref{mass-balance}) \,\textrm {holds}, \, \rho=\hat \rho, \, \bvphi=\hat{\bvphi} \textrm{ and } Q=\hat Q \textrm
{ on } \Gamma_0\}.\nonumber
\end{eqnarray}
The proof of the next theorem follows as that of Theorem \ref{thm:5}, where now the convergence of the sequence of determinants and the positivity of the determinant limit follow from the convexity of $g$.
\begin{theorem} Suppose that the free energy density   in (\ref{Psiform-rho}) is  such that $\hat W(G)= \Psi(|G|)\geq 0$, with $\Psi(\cdot)$ smooth and convex and let $g\neq 0$ satisfy (\ref{g-conv}).
Suppose that the assumptions 
(\ref{frho-conts}), (\ref{frho-lim1}) and (\ref{frho-lim2})  are satisfied.  
Then the total energy admits a minimizer in $\mathcal{A_{\textrm{\begin{tiny}H\end{tiny}}}}$. 
\end{theorem}

\section {Liquid crystal phase transitions in actin networks}

In this section, we apply the energy functions (\ref{Psiform-rho}) to model density dependent liquid crystal phase transitions in actin 
networks.  These are a class of cytoskeletal networks consisting of stiff actin rods jointed
by flexible cross-linkers. Their properties emerge from the interaction of 
liquid rod behavior and network elasticity. 

Parameters to  characterize  these networks include the ratio $\chi=L_a/L_x^0$ of 
the typical  lengths $L_a$ of the actin rod
and that of the the cross-linker, $L_x^0$,  in the reference configuration, the rod aspect ratio $A_a$,  the rod-reference density $\rho_0$, and the reference crosslink density $\sigma_x^0$.
These  are also  the reference parameters used in the  Monte Carlo simulations of these systems by
Bates et al. \cite{BATES} (rigid rod fluids) and Dalhaimer et al.
 \cite{dalhaimer2007crosslinked} (crosslinked rigid rod networks).
These works together with the studies of lyotropic liquid crystals by Kuzuu and Doi \cite{kuzuu-doi1983} motivate the constitutive assumptions in our continuum mechanics treatment.  We assume that  the  molecular interactions responsible for nematic phases in rod-like fluids compete with the solid-like elastic forces due to network crosslinking. We formulate this assumption in terms of the relative energy scales of fluid and solid systems. 
Our goal is to adopt the simplest possible set of assumptions capable of explaining the phase transition behavior.

 In  \cite{dalhaimer2007crosslinked}, the authors argued that  the actin network can be classified
into three regimes according to values of  the ratio $\chi=L_a/L_x^0$, 
where $L_a$ is the length of the actin fiber,
and $L_x^0$ is the typical length of the cross-linker in the reference configuration,  as shown
in Figure \ref{fig:actin-3cases}. 
They found that,
\begin{itemize}
 \item when $\chi<\chi_l$, the network is isotropic in the stress-free state
(external force $\Sigma=0$)
 or under expansion ($\Sigma>0$), and the network will become nematic
 under large enough compression ($\Sigma \leq \Sigma_c < 0$);
 \item when $\chi_l<\chi<\chi_r$, the network is nematic in the stress-free
state (external force $\Sigma=0$),
 or under compression ($\Sigma<0$), and the network will become isotropic
 under large enough expansion ($\Sigma \geq \Sigma_c > 0$);
 \item when $\chi>\chi_r$, the network is nematic regardless 
 the type of applied force.
\end{itemize}
 \begin{figure}[htbp]
 \centering
  \includegraphics[scale=0.26]{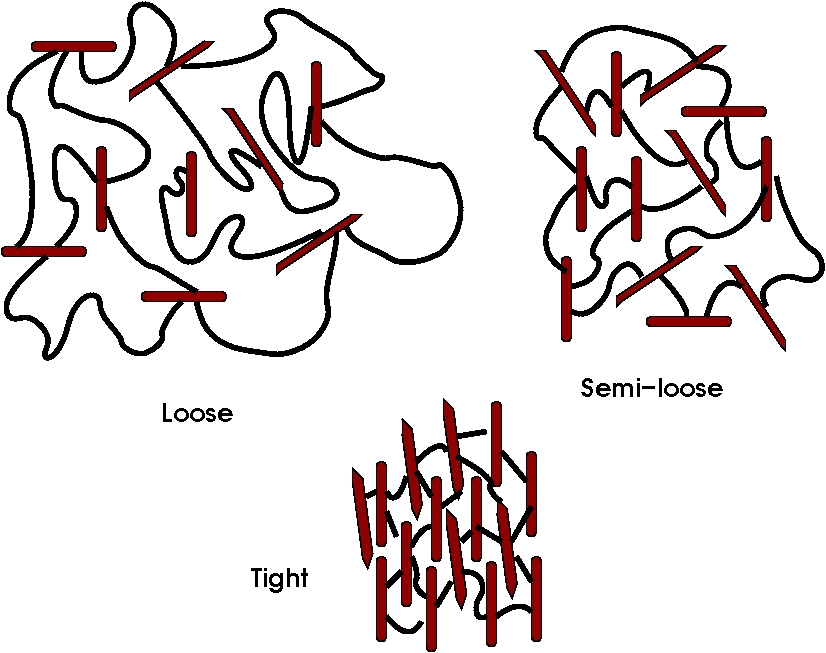}  
 \caption{Three types of actin network according to $\chi$.  Reproduced from
\cite{dalhaimer2007crosslinked}.}  \label{fig:actin-3cases}
 \end{figure}
\hspace{0.1cm}

\subsection{Parameters of the model and assumptions}
We assume that the system is characterized by 
\begin{enumerate}
\item 
The reference configuration $\Omega\subset{\mathbf R}^3$ and  the previously defined positive quantities  $L_x^0$, $ \rho_0$, $\sigma_x^0$, $\chi$ and $A_a$.
\item The energy scaling parameters
\begin{equation}
\mu=RT\sigma_x^0, \quad \nu= RT A_a\rho_0, \label{parameters}
\end{equation}
where $R$ is  the gas constant and $T$ the absolute temperature. 
In particular, they reflect the property that an increase in the  rod aspect ratio, while holding the other parameters fixed,  tends to favor nematic equilibrium. 
\end{enumerate}
\begin{remark}
 Since the rods are  not randomly located in space  as in the case of a fluid 
but serve as crosslink sites,  $\rho_0$ and $L_x^0$ are not independent. For
systems such that $L_x^0>> L_a$, the following estimate  holds:
\begin{equation}
\rho_0=\frac{{\textrm{total volume of rods}}}{\textrm{total undeformed volume}}=
  K\frac{\chi^3}{\mathcal A_a^2}.
\label{rho0-loose}
\end{equation}
We have taken the material of the rod as having mass density 1.  $K$ is a
network constant that accounts for the number of crosslinks per actin unit and
the coordination number of the network. Moreover, in estimating the denominator,
we have assumed that the total volume of the system is fully spanned by the
network. 
\end{remark}

\subsubsection{Nematic rod fluid}
We assume that $f:(-\frac{1}{2}, 1)\times(0,\infty) \longrightarrow \mathbf R$ represents a uniaxial bulk energy, parametrized by $\chi>0$, so that:
\begin{enumerate}
\item There exists a critical value  $\chi_t$, such that, for $0<\chi<\chi_t$, $f$ has two local minima 
$\{s=0, \rho=\rho_0\}$ and $\{s^*>0, \rho^*> \rho_0\}$. For $\chi>\chi_t$,  only the nematic minimum remains.
\item There exists a critical value $\chi_l<\chi_t$ such that,  
\begin{eqnarray}
 && f(0, \rho_0; \chi)<f(s^*, \rho^*; \chi), \quad {\textrm{ for}} \,\, 0<\chi<\chi_l,\\
&& f(0, \rho_0; \chi)>f(s^*, \rho^*; \chi), \quad {\textrm {for}} \,\, \chi_l<\chi<\chi_t,\\
&&f(0, \rho_0; \chi_l)=
f(s^*, \rho^*; \chi_l). \end{eqnarray} 
\item   $f(s^*, \rho^*; \chi)$ decreases with
increasing $\chi$,   and $s^* $ increases and $\rho^*$  decreases, also with respect to $\chi$. 
\item  $f$ has a maximum at $s=s^{**},
\rho=\rho^{**}$,  $0<s^{**}<s^*, $ $\rho_0<\rho^{**}<\rho^{*}$.
\item  It satisfies growth conditions with respect to $s$ and $\rho$:
\begin{eqnarray}
&& \lim_{s\to\{-\frac{1}{2}, 1\}} f(s, \rho; \chi)= +\infty,\, \textrm{for
all}\,\, \rho >0,\\
&&\lim_{\rho\to\{0,\infty\}}f(s, \rho; \chi) =+\infty, \, \textrm{ for all} \, s
\in(-\frac{1}{2}, 1).
\end{eqnarray}\end{enumerate}
In the next section, we provide a method of construction a function $f$  satisfying these properties. 

\begin{figure}[!htp]
 \centering
 \scalebox{.25}{
 \includegraphics{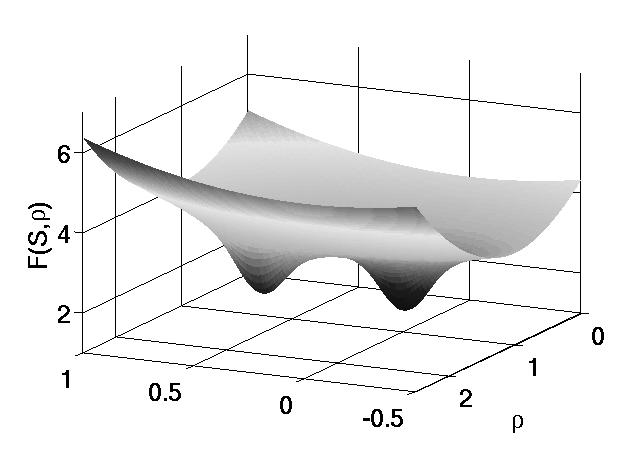}\quad \quad \quad \quad 
 \includegraphics{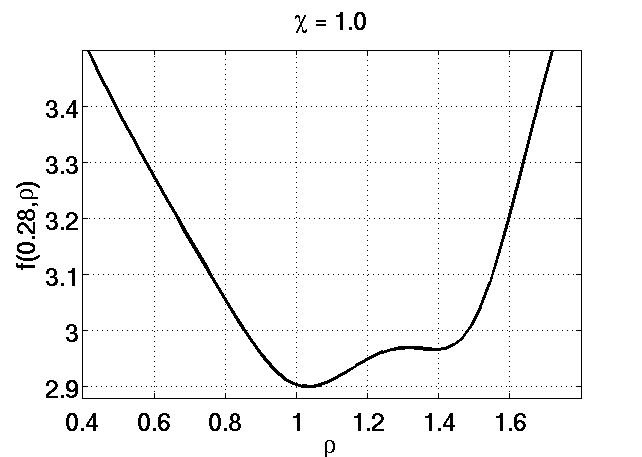}}
 \caption{3D plot of the bulk free energy (\emph{left}) and its corresponding cross section at $s = 0.28$ (\emph{right}) in the case $\chi = 1.0$.}  \label{fig:chi13D}
\end{figure}

\subsection{Density dependent phase transitions}
Let us consider deformations with gradient
\begin{equation}
 F= \diag(\lambda, \lambda, \lambda), \quad \rho \lambda^3=\rho_0. \label{expansion}
\end{equation}
Set $\bn_0 = 0$,  choose $W=\wbtw$ as in (\ref{btw})  and calculate the total energy density
\begin{equation}
 \mathcal E:=  \lambda^3\big(\mu(1-\alpha s^2) + \nu f(s,\rho;\chi)\big).\label{energy-expansion}
\end{equation}
We consider $\mathcal E$ parametrized by $\rho$ and calculate the critical points
\begin{equation}
 \frac{d\mathcal E}{ds}= \lambda^3\big(-2\alpha\mu s+ \nu f_s(s, \rho; \chi)\big)=0.\label{critical-points}
\end{equation}
We now discuss   the solvability of  the critical point equation as the parameter $\rho$ varies.  In the case of multiple solutions, we choose that with the lowest energy. We summarize the results as follows.
\begin{proposition} Let $\rho_0>0$ be prescribed.  Then the homogeneous minimizers of the energy have the following properties:
\begin{enumerate}
 \item For $\chi\geq\chi_t$, the minimizer $s=s(\rho, A_a, \chi)>0$ for all $\rho>0$ with $s(\rho,A_a,\cdot)$ increasing and such that $s(\rho,A_a, \chi)\to 1$ as $\chi\to\infty$. 
\item For $0\leq\chi\leq \chi_t$, there exists a   function $\rho=R(A_a, \chi)$, decreasing as  $A_a$ increases, with $\chi$ held fixed, and increasing as $\chi$ decreases, with 
$A_a$ held fixed, and 
 such that the minimizers satisfy
\begin{eqnarray}
&& s(\rho, A_a,\chi)>0, \,\, {\textrm{for}} \,\, \rho\geq R(A_a, \chi),\\
&& s(\rho,A_a,\chi)\approx 0, \,\, {\textrm{for}} \,\, \rho< R(A_a, \chi).
\end{eqnarray}
Moreover, $R(A_a, \chi)\to \infty$ as $\chi\to 0$. 
Furthermore,  $s(\cdot, A_a, \chi)$ may be discontinuous across $R(\cdot)$. \end{enumerate}
\end{proposition}
Comparing the liquid crystal behavior of the system under expansion with the next simulations on extension    provides additional information on network effects.
\subsubsection{Isotropic-Nematic Phase Transitions}
We now carry out numerical simulations to describe the phase transition behavior  under plane strain deformation given  by 
\begin{eqnarray}
&&F=\diag(\lambda, \lambda, 1), \, \, \, \lambda^2\rho=\rho_0.\label{stretch-biaxial}
\end{eqnarray}

We present three types of plots: the phase diagrams  (\ref{fig:chi1}) and (\ref{fig:chi9-001}) in the $(\rho, A_a)$-plane, the graphs  (\ref{fig:lubinski-like}) and (\ref{fig:order_param_comparison}) of the  equilibrium order parameter $s$ in terms of the  extension  ratio $\lambda$,  and the  stress-strain diagrams (\ref{fig:stress_strain_comparison}).

 The phase diagrams are obtained by solving the equation of critical points, that is,  the analog of (\ref{critical-points}) and, in the case of multiple solutions,  plotting that with smallest energy.  Specifically,   let us define, 
the \emph{isotropic} $\mathcal{E}_{iso} = \mathcal{E}(s=0,\rho)$ and the \emph{nematic}  
$\mathcal{E}_{nema} = \mathcal{E}(s\ne0,\rho)$ energies, respectively. 
The construction of the phase diagrams is summarized as follows:
\begin{enumerate}
 \item Construct the bulk energy function $f(s,\rho; \chi)$ for the problem. 
 \item Define a domain $\mathcal B = [\rho_1,\rho_2]\times[A_{a1},A_{a2}]$ in the \emph{density}-\emph{aspect-ratio} space. 
 \item Choose a discrete subset $\mathcal B_h \subset \mathcal B$ such that,
  \begin{eqnarray*}
   \mathcal B_h  &=& \{ (\rho_{1}+ih_1,A_{a1}+jh_2) \ | \ 0 < h_1 < (\rho_2-\rho_1), 0 < h_2 < (A_{a2}-A_{a1}) \\
   && \& \ i,j \in \mathbb{N}\}.
  \end{eqnarray*}
  \item Given a point $(\rho_i, A_{ai}) \in \mathcal B_h$ compute $s$  by solving the equilibrium equation $\frac{d\mathcal E}{ds}=0$.
  \item The point $(\rho_i, A_{ai})$ 
is labelled  \emph{isotropic}  if  $ \mathcal{E}_{iso}< \mathcal{E}_{nema} $  and \emph{nematic} otherwise.
   \item Finally,  we plot the nematic and isotropic points in a $(\rho, A_a)$-diagram. We follow the   convention of assigning  {\it red }  to nematic points,  and  {\it blue} to  isotropic ones. 
\end{enumerate}
We construct $f$ as follows. Let
  $z:=\det F$, and define
\begin{eqnarray}
\nonumber
h(s,z;\chi,s_{i,n},z_{i,n},\eta_{i,n}) &=& \chi W_{iso}(s,z;s_{i},z_{i},\eta_{i}) +W_{nema}(s,z;s_{n},z_{n},\eta_{n}) \\ \label{eqn:f}
     &&+W_{gr}(s,z),  \\ \nonumber
 W_{iso}(s,z;s_{i},z_{i},\eta_{i}) &=& \arctan{\left(\eta_{i}( (s-s_i)^2+ (z-z_i)^2)\right)} +(s-s_i)^2\\ \label{eqn:f-isotropic}
                            &&+ (z-z_i)^2, \\ \label{eqn:f-nematic}
 W_{nema}(s,z;s_{n},z_{n},\eta_{n}) &=& \arctan{\left(\eta_{n}( (s-s_n)^2+ (z-z_n)^2\right))}, \\ \label{eqn:f-growth}
 W_{gr}(s,z) &=& -(\log(z)+\log(|s-1|(s+0.5))  +z^2. 
\end{eqnarray}
The parameters $s_{i,n},z_{i,n}$ represent 
the position of the \emph{isotropic} and \emph{nematic} minimum,  respectively, and $\eta_{i,n}$  represent the width of  the corresponding well. 
 For a fixed set of parameters  $\{s_{i,n},z_{i,n}, \eta_{i,n}\}$,  let
 \begin{equation} f(s, \rho; \chi)= h(s,z;\chi,s_{i,n},z_{i,n},\eta_{i,n}). \label{f-h} \end{equation}

Figures \ref{fig:chi1}
 and \ref{fig:chi9-001} show phase  diagrams for different values of $\chi$  and  contour plots of $f(s,\rho; \chi)$.  A main feature of these diagrams is that 
 the density at which the nematic phase occurs increases with either  lowering $\chi$ or $A_a$. (We hold  $A_a$ fixed, in the first case, and $\chi$ in the latter).  Moreover, in these diagrams, the  isotropic phase is always present.  It would require values  $\chi>>10^3$ to encounter the  nematic phase only.  We stipulate that imposing a steeper growth of the energy with respect to $\rho$, for $\rho$ large, would also yield purely nematic phase diagrams for $\chi=O(10^3)$. 
   
 
\begin{figure}[!htp]
 \centering
 \includegraphics[scale=0.27]{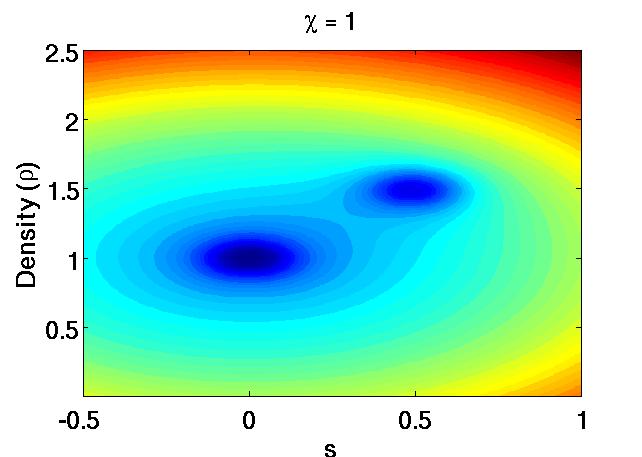} 
 \includegraphics[scale=0.18]{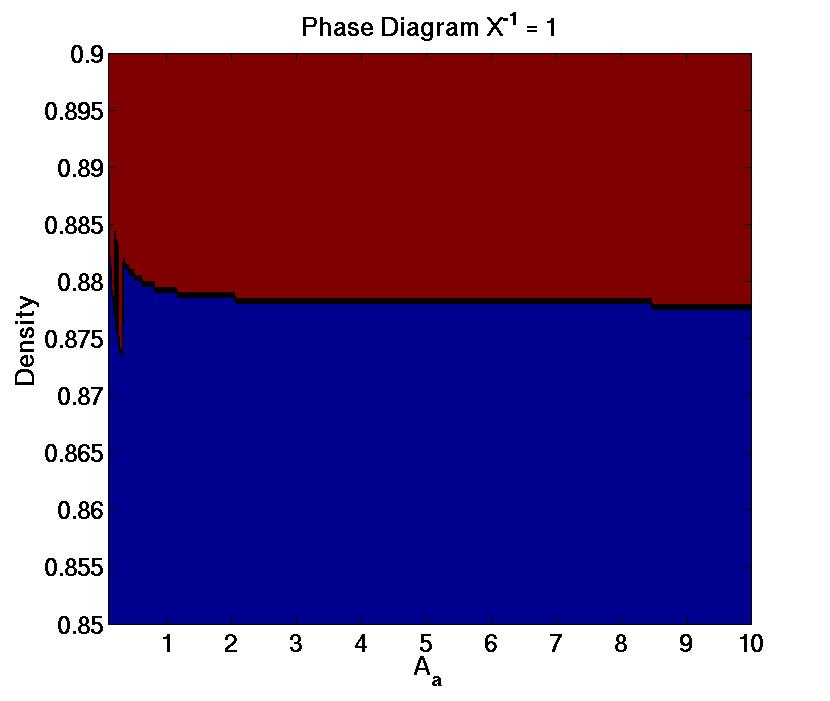}
 \caption{Contour lines for the bulk potential function (\emph{left}) and phase space diagram (\emph{right}) for $\chi = 1$.} 
 \label{fig:chi1}
\end{figure}

\begin{figure}[!htp]
\centering
\includegraphics[scale=0.25]{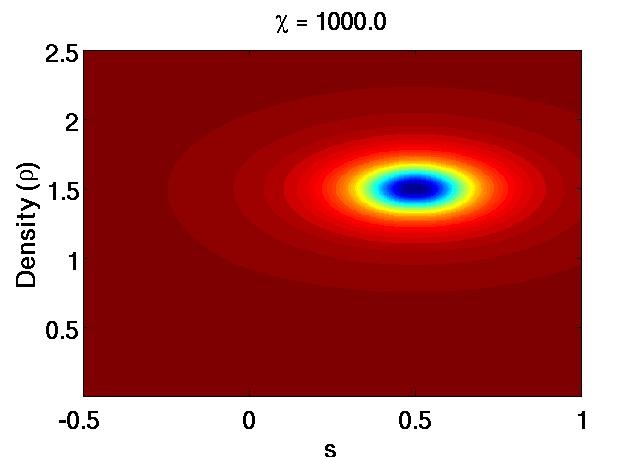}
\includegraphics[scale=0.25]{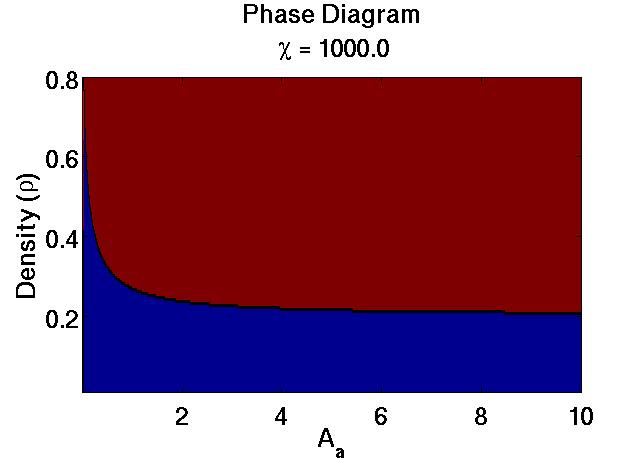}
 \caption{Contour lines for the bulk potential function (\emph{left}) and phase space diagram (\emph{right}) $\chi = 1000.0$.}  \label{fig:chi9-001}
\end{figure}

\subsubsection {Order Parameter Diagrams and Stress-Strain Plots} 

Figures  \ref{fig:order_param_comparison0} and \ref{fig:order_param_comparison} represent  plots of the uniaxial order parameter $s$ with respect to the rod-density $\rho$, for $\chi= 0.5, 3.5, 10 $ and $ 80$,  and for values of the aspect ratio $A_a$ ranging from $0.01$ to $80$. These values represent a range of shapes, from  oblate  cylinders to very elongated rods. The first  graph in figure 4.5 presents  two density-intervals with distinguished behavior,  one corresponding to well aligned rods at high density, with a drop in the uniaxial order parameter as the density decreases to a critical value, and a second interval of  further decrease in $s$ as $\rho$ continues decreasing.   These graphs are in full agreement with those obtained by Montecarlo simulations in \cite{BATES} and \cite{dalhaimer2007crosslinked}.
Moreover, the second graph of figure 4.5 and those in 4.6 present a third density interval  of increase of  the order parameter. This is due to the rod alignment that results from larger extension ratios (i.e.,  smaller densities), and it is a consequence of the elastic network connections of the rods.  Proposition (4.1) shows that this behavior is not analytically predicted when subjecting the material to uniform expansion. It is not reported either in \cite{dalhaimer2007crosslinked}.  Another feature that emerges when comparing the two graphs on the right hand sides of figures  4.5 and 4.6 is that,  for  larger $\chi$,  it requires to reach a lower density to increase the rod alignment. This may indicate the additional  difficulty in aligning larger rods, in comparison with smaller ones. 

We point out that  the first graph in  figure \ref{fig:order_param_comparison} shows the existence of oblate phases of rods with small aspect ratio, in the order of $10^{-2}$. This behavior is presented by cytoskeletal networks of red blood 
cells \cite{RBC2012,CAO2011,CELLS:ACTIN,RBC1990}.  

\begin{figure}[!htp]
 \centering
\scalebox{.23}{
\includegraphics{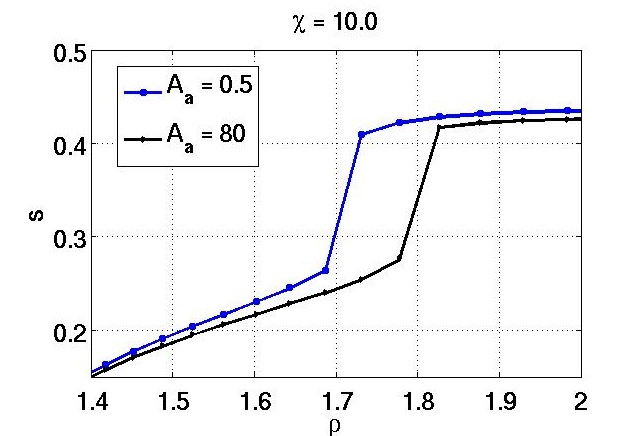}\quad \quad 
 \includegraphics{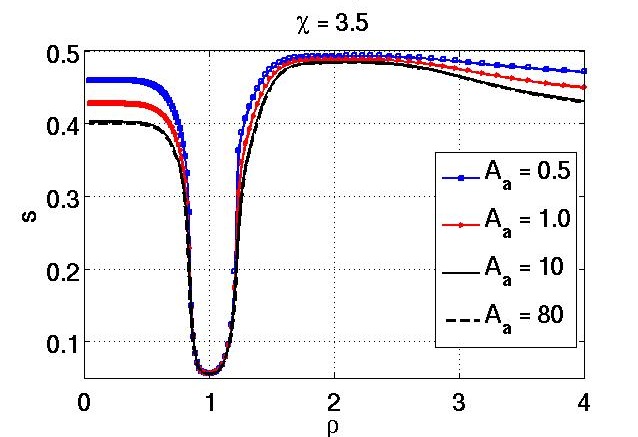} }
 \caption{Order parameter Vs density for $\chi = 10$ (\emph{left}) and $\chi = 3.5$ (\emph{right}).  Nematic well: $s=0.5$, $\rho=1.5$.}  \label{fig:order_param_comparison0}
\end{figure}

\begin{figure}[!htp]
 \centering
\scalebox{.26}{
 \includegraphics{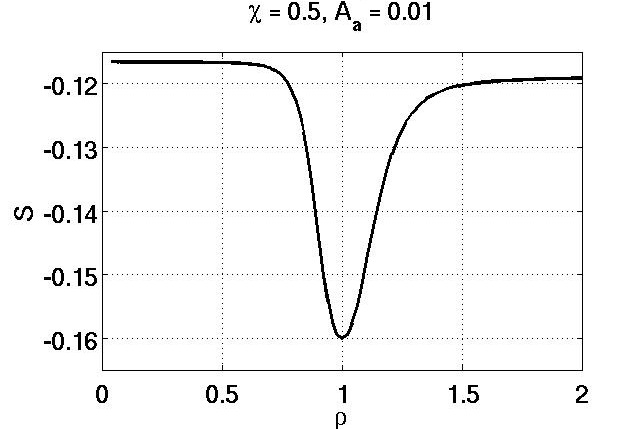}
 \includegraphics{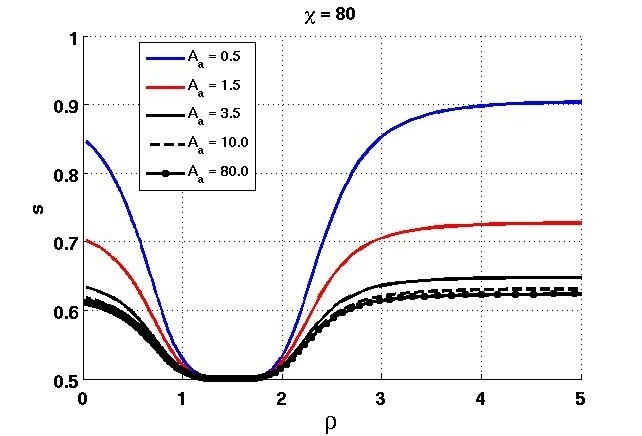} }
 \caption{Order parameter Vs density for $\chi = 0.5$ (\emph{left}) and $\chi = 80$ (\emph{right}). Nematic well: $s=0.5$, $\rho=0.5$.}  \label{fig:order_param_comparison}
\end{figure}

We conclude  this section discussing  the  stress-strain Figure \ref{fig:stress_strain_comparison}.  For $\chi=0.5$ and for  small and medium values of the rod aspect ratio,  the stress-strain curves are monotonic and present a {\it soft} region followed by a steeper growth. However,  we find that  for $A_a=80$,  the stress-strain curves are non-monotonic.  
The change of monotonicity occurs precisely where the order parameter experiences a sharp increase or decrease, indicating the change of volume accompanying  rod order rearrangement. However, we also found   shallower non-monotonic profiles, including for systems experiencing the nematic-isotropic phase transition,  for aspect ratios  smaller  than 80.  This seems to indicate that realignment of rods with  large aspect ratio affects change of volume  in a more significant way than for the smaller counterparts. 

 elastic part.

\begin{figure}[!htp]
\centering
\scalebox{.26}{
\includegraphics{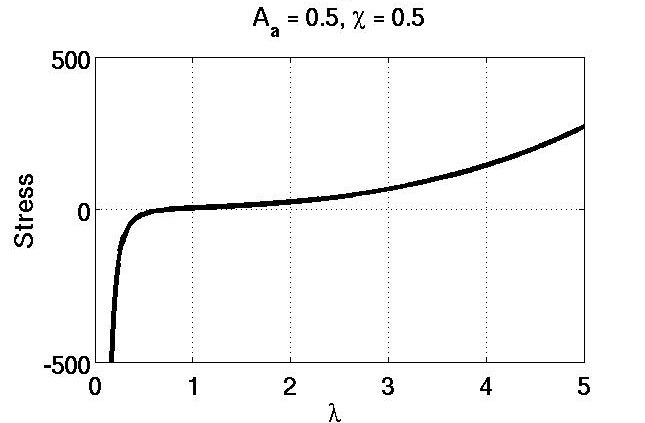}\quad \quad 
 \includegraphics{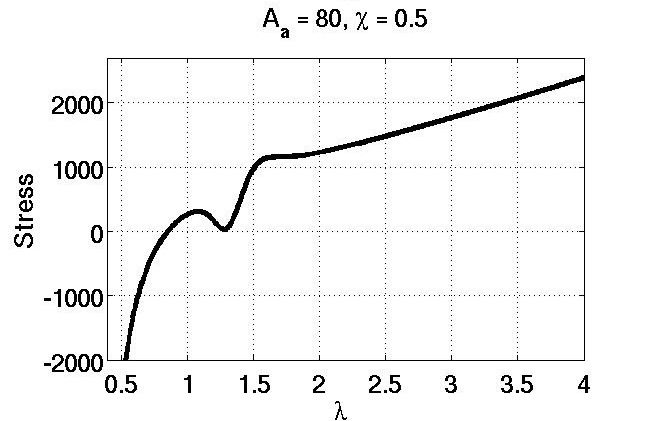} }
 \caption{Plots of the $xx$-components of the first Piola-Kirchhoff stress tensor. } 
\end{figure}

\section{Conclusions}
We have presented and analyzed models of anisotropic elasticity based on the theory of liquid crystal elastomers, and applied  them to  modeling order phase transitions in actin  networks.
 
We followed a strategy to show existence of minimizers based on the  theory of isotropic nonlinear 
elasticity. This required  assuming that the energy density function is polyconvex  with respect to the anisotropic deformation tensor 
$G=L^{-\frac{1}{2}}\nabla\bvphi L_0^{\frac{1}{2}}$. The latter is at the core  of the works on liquid crystal elastomers by Warner and Terentjev.  
 An essential ingredient in the 
analysis is the assumption of a constitutive equation relating the shape of the polymer represented by the tensor $L$ with the order tensor $Q$ 
of the liquid crystal  rigid units. The linear constitutive relation involves the restriction that both tensors become singular 
in the same region of the order parameter space. 
However, the linear relation  implicitly involves the constraint of the trace of $L$ being constant, and therefore, it restricts 
the value of the sum of the principle axis of the ellipsoid associated with $L$. In future works, we will explore how to avoid this 
restriction by assuming nonlinear relations between the two tensors. 

The assumptions on the bulk uniaxial free energy density function $f$,  as well as the algorithm  to  generate specific forms of it,  follow earlier works on uniaxial lyotropic liquid crystals. They also 
incorporate physically meaningful  growth conditions  required in the analysis.  The results on phase transitions that we  obtained from the proposed continuum theory  show  good  agreement with  those stemming from the molecular simulations that motivated this work, and from experimental results.  In forthcoming work, we aim at constructing 
bulk free energy functions based on the Onsager rigid-rod theory. Although there is limited information on   temperature dependence of the Landau-de Gennes energy for thermotropic nematic liquid crystals (\ref{bulk-thermo}),  the dependence on concentration in the lyotropic case seems to be lacking. In future work, we will also address the behavior of the system under shearing and consider the 
case of periodic crosslinking, also found in some actin networks.

 \section{Acknowledgment} This research was partially supported by a grant from the National Science Foundation, NSF-DMS 0909165.
 
\bibliography{arxiv}
\bibliographystyle{plain}
\end{document}